\documentclass[12pt, reqno]{amsart}
\usepackage{amsmath, amsthm, amscd, amsfonts, amssymb, graphicx, color}
\usepackage[bookmarksnumbered, colorlinks, plainpages]{hyperref}
\hypersetup{colorlinks=true,linkcolor=red, anchorcolor=green, citecolor=cyan, urlcolor=red, filecolor=magenta, pdftoolbar=true}

\newtheorem{theorem}{Theorem}[section]
\newtheorem{lemma}[theorem]{Lemma}
\newtheorem{proposition}[theorem]{Proposition}
\newtheorem{corollary}[theorem]{Corollary}
\theoremstyle{definition}
\newtheorem{definition}[theorem]{Definition}

\newtheorem{assumption}[theorem]{Assumption}

\theoremstyle{remark}
\newtheorem{remark}[theorem]{Remark}

\numberwithin{equation}{section}

\def\rr{{\mathbb R}}
\def\rn{{{\rr}^n}}
\def\zz{{\mathbb Z}}
\def\cc{{\mathbb C}}
\def\nn{{\mathbb N}}

\def\ca{{\mathcal A}}

\def\cf{{\mathcal F}}

\def\cm{{\mathcal M}}

\def\cp{{\mathcal P}}

\def\cs{{\mathcal S}}

\def\cx{{\mathcal X}}

\def\fz{\infty}
\def\az{\alpha}
\def\bz{\beta}

\def\bdz{\Delta}

\def\bgz{{\Gamma}}

\def\lz{\lambda}

\def\tz{\theta}

\def\lf{\left}
\def\r{\right}

\def\hs{\hspace{0.25cm}}
\def\ls{\lesssim}

\def\ov{\overline}
\def\noz{\nonumber}
\def\wz{\widetilde}
\def\wh{\widehat}
\def\st{\subset}
\def\com{\complement}

\def\dist{\mathop\mathrm{\,dist\,}}
\def\supp{\mathop\mathrm{\,supp\,}}
\def\atom{\mathop\mathrm{\,atom\,}}

\def\div{\mathop\mathrm{div}}
\def\essinf{\mathop\mathrm{\,ess\,inf\,}}
\def\esup{\mathop\mathrm{\,ess\,sup\,}}

\def\vp{{L^{p(\cdot)}(\rn)}}
\def\hp{{H^{p(\cdot)}(\rn)}}
\def\vhp{H_L^{p(\cdot)}(\rn)}

\def\atom{\mathbb{H}_{L,\,{\rm at}}^{p(\cdot),\,r}(\rn)}
\def\ath{H_{L,\,{\rm at}}^{p(\cdot),\,r}(\rn)}

\def\uz{\underline}

\def\rnn{\rr_+^{n+1}}
\def\dydt{\,\frac{dy\,dt}{t^{n+1}}}

\begin{document}

\setcounter{page}{1}

\title[Atomic Characterization of Hardy Spaces]
{Atomic Characterizations of Hardy Spaces
Associated to Schr\"{o}dinger Type Operators}

\author[J. Zhang \MakeLowercase{and} Z. Liu]
{Junqiang Zhang$^{1*}$ \MakeLowercase{and} Zongguang Liu$^{2}$}

\address{$^{1}$School of Science, China University of Mining and Technology-Beijing,
Beijing 100083, People's Republic of China}
\email{\textcolor[rgb]{0.00,0.00,0.84}{jqzhang@cumtb.edu.cn;
zhangjunqiang@mail.bnu.edu.cn}}

\address{$^{2}$School of Science, China University of Mining and Technology-Beijing,
Beijing 100083, People's Republic of China}
\email{\textcolor[rgb]{0.00,0.00,0.84}{liuzg@cumtb.edu.cn}}


\let\thefootnote\relax\footnote{Copyright 2018 by the Tusi Mathematical Research Group.}



\date{Received: xxxxxx; Revised: yyyyyy; Accepted: zzzzzz.
\newline \indent $^{*}$Corresponding author}

\begin{abstract}
In this article, the authors consider the Schr\"{o}dinger type
operator $L:=-{\rm div}(A\nabla)+V$ on $\mathbb{R}^n$ with $n\geq 3$,
where the matrix $A$ is symmetric and satisfies
uniformly elliptic condition and the nonnegative potential
$V$ belongs to the reverse H\"{o}lder class $RH_q(\mathbb{R}^n)$
with $q\in(n/2,\,\infty)$.
Let $p(\cdot):\ \mathbb{R}^n\to(0,\,1]$ be a variable exponent function
satisfying the globally $\log$-H\"{o}lder continuous condition.
The authors introduce the variable Hardy space $H_L^{p(\cdot)}(\mathbb{R}^n)$ associated to $L$
and establish its atomic characterization.
The atoms here are closer to the atoms of
variable Hardy space $H^{p(\cdot)}(\mathbb{R}^n)$ in spirit,
which further implies that $H^{p(\cdot)}(\mathbb{R}^n)$ is continuously embedded in
$H_L^{p(\cdot)}(\mathbb{R}^n)$.
\end{abstract} \maketitle

\section{Introduction and main results}

The theory of classical real Hardy spaces $H^p(\rn)$ was first
introduced by Stein and Weiss \cite{sw60} in the early 1960s, and
then was systematically developed by Fefferman and Stein \cite{fs72}.
One important feature of $H^p(\rn)$ is that it has so called
``atomic characterization'' (see \cite{c74,l78}).
This feature allows one to reduce the study of the
boundedness of operators on $H^p(\rn)$ to studying their behaviours on
single atoms.
As a generalization of classical Hardy spaces,
Nakai and Sawano \cite{ns12} introduced variable Hardy spaces $H^{p(\cdot)}(\rn)$,
established their atomic characterizations and
investigated their dual spaces.
Independently, Cruz-Uribe and Wang \cite{cw14} also studied the
variable Hardy spaces $H^{p(\cdot)}(\rn)$ with $p(\cdot)$ satisfying some conditions
slightly weaker than those used in \cite{ns12}.

On the other hand, notice that $H^p(\rn)$ is essentially associated with the Laplace
operator
$\bdz:=\sum_{j=1}^n\frac{\partial^2}{\partial x_j^2}$
(see, for example, \cite{dy051}).
In recent years, there has been a lot of attention paid to the
study of Hardy spaces associated with different operators, which has been
a very active research topic in harmonic analysis.
First, Auscher et al. \cite{adm05}, and then Duong and Yan
\cite{dy05,dy051}, introduced Hardy and BMO spaces associated with an operator
$L$ whose heat kernel has a pointwise Gaussian upper bound.
Then, Hofmann et al. \cite{hlmmy11} studied Hardy spaces associated with non-negative
self-adjoint operator $L$ satisfying the Davies-Gaffney estimates
and established its atomic characterization.
Different from the classical atoms, the atoms in \cite{hlmmy11}
are defined by replacing the standard vanishing moment condition into
a matter of membership in the range of the operator $L$.
Since then, much work has been done in the real variable theory of
Hardy spaces associated with different operators, see, for example,
\cite{hm09,jy10,y08}.
In particular, Dziuba\'{n}ski and Zienkiewicz \cite{dz99,dz02,dz03}
established an new atomic characterization of Hardy spaces $H_L^p(\rn)$
associated to Schr\"{o}dinger operator $L:=-\Delta+V$,
where $V$ belongs to a certain reverse H\"{o}lder class.
Compared with atoms defined in \cite{hlmmy11},
the atoms introduced by Dziuba\'{n}ski and Zienkiewicz are very different
and closer to the classical atoms of $H^p(\rn)$ in spirit,
since they also have some ``local" vanishing condition.

Recently, motivated by \cite{dz99,dz02,dz03},
Cao et al. \cite{ccyy14} established an atomic characterization
of Musielak-Orlicz-Hardy spaces $H_{\varphi,\,L}(\rn)$
associated to Schr\"{o}dinger operator $L=-\Delta+V$
and, as an application, studied the boundedness of second order Riesz transform $\nabla^2L^{-1}$
on $H_{\varphi,\,L}(\rn)$.
The atomic characterization of $H_{\varphi,\,L}(\rn)$ in \cite{ccyy14}
is similar to those of \cite{dz99,dz02,dz03},
however, Cao et al. used a totally different method.
Later, Yang \cite{y14} generalized the results of \cite{ccyy14} to
the Musielak-Orlicz-Hardy spaces associated to
Schr\"{o}dinger type operator $L:=-\div(A\nabla)+V$,
where $A$ is a matrix satisfying the uniform elliptic condition and
$V$ belongs to a certain reverse H\"{o}lder class.

In this paper, motivated by \cite{ccyy14,y14,yzz15},
we consider the Schr\"{o}dinger type operator
\begin{equation}\label{eq o}
L:=-\div(A\nabla)+V\ \ \text{on}\ \ \rn,\,n\geq 3,
\end{equation}
where $V$ is a nonnegative potential and $A:=\{a_{ij}\}_{1\le i,j\le n}$
is a matrix of measurable functions satisfying the following condition:
\begin{assumption}\label{as 1}
There exists a constant $\lz\in(0,1]$ such that,
for any $x,\,\xi\in\rn$,
$$a_{ij}(x)=a_{ji}(x)\quad \text{and}\quad
\lz|\xi|^2\le\sum_{i,j=1}^n a_{ij}(x)\xi_i\xi_j\le\lz^{-1}|\xi|^2.$$
\end{assumption}

Let $L=-\div(A\nabla)+V$ be as in \eqref{eq o},
where $V$ is a nonnegative potential on $\rn$
with $n\geq 3$ and belongs to the reverse H\"{o}lder class
$RH_q(\rn)$ for some $q\in(n/2,\fz)$.
In this paper, we establish an atomic characterization
of variable Hardy space $H_L^{p(\cdot)}(\rn)$ associated to $L$.
As an application, we show that the variable Hardy space $H^{p(\cdot)}(\rn)$
is continuously embedded in $H_L^{p(\cdot)}(\rn)$.
To state the main results, we first
recall some notation and definitions.

Let $\cp(\rn)$ be the set of all the measurable functions $p(\cdot):\ \rn\to (0,\,\fz)$
satisfying
\begin{equation}\label{eq var}
p_-:=\essinf_{x\in\rn}p(x)>0\ \ \text{and}\ \ p_+:=\esup_{x\in\rn}p(x)<\fz.
\end{equation}
A function $p(\cdot)\in\cp(\rn)$ is called a \emph{variable exponent function on $\rn$}.
The \emph{variable Lebesgue space $\vp$}
is a generalization of classical Lebesgue space,
via replacing the constant exponent $p$ by a variable exponent function
$p(\cdot):\ \rn\to(0,\,\fz)$,
which consists of all measurable functions $f$ on $\rn$ such that,
for some $\lz\in(0,\,\fz)$,
$\int_\rn [|f(x)|/\lz]^{p(x)}\,dx<\fz$,
equipped with the \emph{Luxemburg} (or known as the \emph{Luxemburg-Nakano}) {(quasi-)norm}
\begin{equation}\label{eq norm}
\|f\|_{\vp}:=\inf\lf\{\lz\in(0,\,\fz):\ \int_\rn
\lf[\frac{|f(x)|}{\lz}\r]^{p(x)}\,dx\le 1\r\}.
\end{equation}

Recall that a measurable function $g\in\cp(\rn)$ is said to be
\emph{globally $\log$-H\"{o}lder continuous},
denoted by $g\in C^{\log}(\rn)$, if there exist constants $C_1,\,C_2\in(0,\,\fz)$
and $g_\fz\in\mathbb{R}$ such that, for any $x,\,y\in\rn$,
\begin{equation*}
|g(x)-g(y)|\le \frac{C_1}{\log(e+1/|x-y|)}
\end{equation*}
and
\begin{equation*}
|g(x)-g_\fz|\le \frac{C_2}{\log(e+|x|)}.
\end{equation*}

The \emph{quadratic operator} $S_L$, associated to $L$, is defined by
setting, for any $f\in L^2(\rn)$ and $x\in\rn$,
\begin{equation}\label{eq quadratic}
S_L(f)(x):=\lf[\iint_{\bgz(x)}\lf|t^2Le^{-t^2L}(f)(y)\r|^2\dydt\r]^{1/2},
\end{equation}
where $\Gamma(x):=\{(y,\,t)\in\rn\times(0,\,\fz):\ |y-x|<t\}$
denotes the cone with vertex $x\in\rn$.

\begin{definition}\label{def vhp}
Let $p(\cdot)\in \cp(\rn)$ satisfy $p_+\in(0,\,1]$
and $L$ be an operator as in \eqref{eq o}.
The \emph{variable Hardy space} $H_{L}^{p(\cdot)}(\rn)$ is defined as the
completion of the space
\begin{align*}
\lf\{f\in L^2(\rn):\ \|S_L(f)\|_{\vp}<\fz\r\}
\end{align*}
with respect to the \emph{quasi-norm} $\|f\|_{H_L^{p(\cdot)}(\rn)}:=\|S_L(f)\|_{\vp}$.
\end{definition}

Motivated by \cite{ccyy14,y14}, we define the atoms associated to $L$ as follows.
\begin{definition}\label{def atom}
Let $p(\cdot)\in\cp(\rn)$ with $p_+\in(0,1]$ and $m(\cdot,V)$ be as in \eqref{eq aux}.
For any given $r\in(1,\fz]$, a measurable function $a$ on $\rn$
is called an \emph{$(p(\cdot),\,r)$-atom} associated with the ball $B:=B(x_0,r_0)$ of $\rn$, if
\begin{enumerate}
\item[(i)] $\supp a\st B$;

\item[(ii)] $\|a\|_{L^r(B)}\le |B|^{1/r}\|\chi_B\|^{-1}_{L^{p(\cdot)}(\rn)}$;

\item[(iii)] $\int_{\rn} a(x)\,dx=0$ if $r_0<[m(x_0,V)]^{-1}$.
\end{enumerate}
\end{definition}

\begin{remark}
atoms of Definition \ref{def atom} are closer to atoms of $\hp$
(see Definition \ref{def atom-1} below).
In fact, it is easy to see that, for any $r\in(1,\,\fz]$,
if $a$ is an $(p(\cdot),\,r,\,0)$-atom of $\hp$,
then $a$ is an $(p(\cdot),\,r)$-atom of Definition \ref{def atom}.
\end{remark}

In what follows, for any $p(\cdot)\in\cp(\rn)$ with $0<p_-\le p_+\le 1$,
any sequences $\{\lz_j\}_{j\in\nn}\st\cc$ and $\{B_j\}_{j\in\nn}$ of balls in $\rn$, define
\begin{equation}\label{eq norm-2}
\ca(\{\lz_j\}_{j\in\nn},\,\{B_j\}_{j\in\nn}):=\lf\|\lf\{\sum_{j\in\nn}
\lf[\frac{|\lz_j|\chi_{B_j}}{\|\chi_{B_j}\|_{\vp}}\r]^{p_-}\r\}^{\frac1{p_-}}\r\|_{\vp}.
\end{equation}

\begin{definition}\label{def at-Hardy}
Let $L$ be as in \eqref{eq o}, $p(\cdot)\in \cp(\rn)$ with $p_+\in(0,\,1]$ and $r\in(1,\fz]$.
For a measurable function $f$ on $\rn $,
$f=\sum_{j=1}^\fz \lz_j a_j$ is called an
\emph{atomic $(p(\cdot),\,r)$-representation} of $f$
if $\{a_j\}_{j\in\nn}$ is a family of $(p(\cdot),\,r)$-atoms,
the summation converges in $L^2(\rn)$ and $\{\lz_j\}_{j\in\nn}\st\cc$ satisfies that
$\ca(\{\lz_j\}_{j\in\nn},\,\{B_j\}_{j\in\nn})<\fz$,
where, for any $j\in\nn$, $B_j$ is the ball associated with $a_j$
and $\ca(\{\lz_j\}_{j\in\nn},\{B_j\}_{j\in\nn})$ is as in \eqref{eq norm-2}.
Let
\begin{align}\label{eq at-Hardy}
\atom:=\{f:\ f\ \text{has an atomic}\ (p(\cdot),\,r)\text{-representation}\}.
\end{align}
Then the \emph{variable atomic Hardy space $\ath$} is defined as
the completion of $\atom$ with respect to the
\emph{quasi-norm}
\begin{align*}
\|f\|_{\ath}
&:=\inf\Bigg\{\ca\lf(\{\lz_j\}_{j\in\nn},\,\{B_j\}_{j\in\nn}\r):\ \\
&\qquad\qquad f=\sum_{j=1}^\fz\lz_j a_j\ \text{is an atomic}\
(p(\cdot),\,r)\text{-representation}\Bigg\},
\end{align*}
where the infimum is taken over all the atomic
$(p(\cdot),\,r)$-representations of $f$ as above.
\end{definition}

The following theorem is the main result of this paper, which
establishes the atomic characterization of $\vhp$.
\begin{theorem}\label{thm-2}
Let $L$ be as in \eqref{eq o} with $A$ satisfying Assumption \ref{as 1},
$\mu_0\in(0,1]$ as in \eqref{eq mu}, $p(\cdot)\in C^{\log}(\rn)$ with
$\frac{n}{n+\mu_0}<p_-\le p_+\le1$ and $r\in(1,\,\fz)$ such that $\mu_0+\frac{n}{r}>\frac{n}{p_-}$.
Then $\ath$ and $\vhp$ coincide with equivalent quasi-norms.
\end{theorem}

The proof of Theorem \ref{thm-2} is in Subsection \ref{s3-1}.
Motivated by the proof of \cite[Theorem 2.3]{ccyy14},
we prove Theorem \ref{thm-2} by using the atomic decomposition of variable tent space,
the holomorphic functional calculus of $L$ and some subtle estimates associated to $L$
(see Lemma \ref{lem 2.0} and \eqref{eq 2.16} below).
Besides, comparing with the proof of \cite[Theorem 2.3]{ccyy14},
we also use the finite speed propagation property
of the nonnegative self-adjoint operator $L$,
which, in some sense, simplifies the proof.

\begin{remark}\label{rem-8}
In particular, if $A:=I$ is the identity matrix, namely, $L=-\Delta+V$,
and $p(\cdot)=p\in(\frac{n}{n+\mu_0},1]$ is a constant exponent,
then Theorem \ref{thm-2} completely covers \cite[Theorem 1.11]{dz02}.
\end{remark}

Moreover, as an application of Theorem \ref{thm-2},
we show that variable Hardy space $H^{p(\cdot)}(\rn)$ is
continuously embedded in $\vhp$.
The variable Hardy space $\hp$ is first introduced by
E. Nakai and Y. Sawano \cite{ns12} (independently, by Cruz-Uribe and Wang \cite{cf13}).
Let $\cs(\rn)$ be the \emph{space of all Schwartz functions} and $\cs'(\rn)$
the \emph{space of all Schwartz distributions}. For any $N\in\nn$, define
\begin{align*}
\cf_N(\rn):=\lf\{\psi\in\cs(\rn):\ \sum_{\bz\in\zz_+^n,\,|\bz|\le N}\sup_{x\in\rn}(1+|x|)^N
\lf|D^\bz\psi(x)\r|\le 1\r\},
\end{align*}
where, for any $\bz:=(\bz_1,\,\ldots,\,\bz_n)\in\zz_+^n$,
$|\bz|:=\bz_1+\cdots +\bz_n$ and $D^\bz:=(\frac{\partial}{\partial x_1})^{\bz_1}\cdots
(\frac{\partial}{\partial x_n})^{\bz_n}$.
For any $N\in\nn$, the \emph{grand maximal function} $\cm_F$ is defined by setting,
for any $f\in\cs'(\rn)$ and $x\in\rn$,
\begin{align*}
\cm_N(f)(x):=\sup\{|\psi_t\ast f(x)|:\ t\in(0,\,\fz),\, \psi\in \cf_N(\rn)\},
\end{align*}
where, for any $t\in(0,\,\fz)$ and $\xi\in\rn$, $\psi_t(\xi):=t^{-n}\psi(\xi/t)$.

\begin{definition}[\cite{ns12}]\label{def hardy}
Let $p(\cdot)\in \cp(\rn)$ and $N\in(\frac{n}{p_-}+n+1,\,\fz)$.
Then the \emph{variable Hardy space} $H^{p(\cdot)}(\rn)$ is defined by setting
\begin{align*}
H^{p(\cdot)}(\rn):=\lf\{f\in\cs'(\rn):\ \|f\|_{H^{p(\cdot)}(\rn)}:=\|\cm_N(f)\|_{\vp}<\fz\r\}.
\end{align*}
\end{definition}

\begin{corollary}\label{cor-1}
Let $L$ and $p(\cdot)$ be as in Theorem \ref{thm-2}.
Assume that $\frac{n}{n+\mu_0}<p_-\le p_+\le1$ with
$\mu_0\in(0,1]$ as in \eqref{eq mu}.
Then $H^{p(\cdot)}(\rn)$ is continuously embedded in $H_L^{p(\cdot)}(\rn)$,
namely, there exists a positive constant $C$ such that,
for any $f\in H^{p(\cdot)}(\rn)$,
\begin{align*}
\|f\|_{H_L^{p(\cdot)}(\rn)}\le C\|f\|_{H^{p(\cdot)}(\rn)}.
\end{align*}
\end{corollary}

This article is organized as follows.
In Section \ref{s2}, we give some preliminaries on
the Schr\"{o}dinger operator $L$ and variable
Lebesgue space $L^{p(\cdot)}(\rn)$.
In Section \ref{s3}, we give the proofs of
Theorem \ref{thm-2} and Corollary \ref{cor-1}.

We end this section by making some conventions on notation.
In this article,
we denote by $C$ a positive constant which is independent of the main parameters,
but it may vary from line to line.
We also use $C_{(\az, \bz,\ldots)}$ to denote a positive constant depending on
the parameters $\az$, $\bz$, $\ldots$.
The \emph{symbol $f\ls g$} means that $f\le Cg$.
If $f\ls g$ and $g\ls f$,  we then write $f\sim g$.
Let $\nn:=\{1,\,2,\,\ldots\}$, $\zz_+:=\nn\cup\{0\}$.
For any measurable subset $E$ of $\rn$, we denote by $E^\com$ the
\emph{set $\rn\setminus E$}.
For any $r\in\mathbb{R}$, the \emph{symbol} $\lfloor r\rfloor$
denotes the largest integer $m$ such that $m\le r$.
For any $\mu\in(0,\,\pi)$, let
\begin{align}\label{eq 0.1}
\Sigma_\mu^0:=\lf\{z\in\cc\setminus\{0\}:\ |\arg z|<\mu\r\}.
\end{align}
For any ball
$B:=B(x_B,r_B):=\{y\in\rn:\ |x-y|<r_B\}\st\rn$
with $x_B\in\rn$ and $r_B\in(0,\,\fz)$, $\az\in(0,\,\fz)$ and $j\in\nn$,
we let $\az B:=B(x_B,\az r_B)$,
\begin{align}\label{eq-ujb}
U_0(B):=B\ \ \ \text{and}\ \ \ U_j(B):= (2^jB)\setminus (2^{j-1}B).
\end{align}
For any $p\in[1,\,\fz]$, $p'$ denotes its conjugate number,
namely, $1/p+1/p'=1$.

\section{Preliminaries}\label{s2}

In this section, we recall some notions and
results on the Schr\"{o}dinger type operator $L=-\div(A\nabla)+V$
and variable Lebesgue space $\vp$.

We first recall the definition of the auxiliary function $m(\cdot,V)$
introduced by Shen \cite[Definition 2.1]{sh} and its properties.
Let $V\in RH_q(\rn)$, $q\in(n/2,\,\fz)$, and $V\not\equiv 0$.
For any $x\in\rn$, the auxiliary function $m(x,\,V)$ is defined by
\begin{equation}\label{eq aux}
\frac{1}{m(x,\,V)}:=\sup\lf\{r\in(0,\,\fz):\ \frac{1}{r^{n-2}}\int_{B(x,r)}V(y)\,dy\le 1\r\}.
\end{equation}

For the auxiliary function $m(\cdot,V)$, we have the following Lemma \ref{lem aux},
which is just \cite[Lemma 1.4]{sh}.

\begin{lemma}[\cite{sh}]\label{lem aux}
Let $m(\cdot,V)$ be as in \eqref{eq aux}. Then there exist positive constants
$\wz{C},C$ and $k_0$ such that, for any $x,y\in\rn$,
\begin{enumerate}
\item[(i)] $C^{-1}m(x,\,V)\le m(y,\,V)\le Cm(x,\,V)$ if $|x-y|\le \wz{C}[m(x,\,V)]^{-1}$;

\item[(ii)] $m(y,\,V)\le C[1+|x-y|m(x,\,V)]^{k_0}m(x,\,V)$;

\item[(iii)] $m(y,\,V)\geq C^{-1} [1+|x-y|m(x,\,V)]^{-\frac{k_0}{k_0+1}}m(x,\,V)$.
\end{enumerate}
\end{lemma}

The following lemma is \cite[Theorem 4]{at98}.
\begin{lemma}[\cite{at98}]\label{lem 1-1}
Let $L_0:=-\div(A\nabla)$ with $A$ satisfying Assumption \ref{as 1}
and $\{e^{-tL_0}\}_{t\geq 0}$ the heat semigroup generated by $L_0$.
Then the kernels $h_t(x,y)$ of the heat semigroup $\{e^{-tL_0}\}_{t\geq 0}$ are
continuous and there exists a constant $\az_0\in(0,1]$ such that,
for any given $\az\in(0,\az_0)$,
\begin{equation*}
|h_t(x+h,y)-h_t(x,y)|+|h_t(x,y+h)-h_t(x,y)|\le
\frac{C}{t^{n/2}}\lf[\frac{|h|}{\sqrt{t}}\r]^\az e^{-c\frac{|x-y|^2}{t}},
\end{equation*}
where $t\in(0,\,\fz)$, $x,y,h\in\rn$ with $|h|\le\sqrt{t}$
and $C,c$ are positive constants independent of $t,x,y,h$.
\end{lemma}

The following lemma is \cite[Lemma 2.6]{y14}.
\begin{lemma}[\cite{y14}]\label{lem 1-2}
Let $L$ be as in \eqref{eq o} with $A$ satisfying Assumption \ref{as 1}
and $V\in RH_q(\rn)$, $q\in(n/2,\,\fz)$.
Assume that $K_t$ is the kernel of the heat semigroup $\{e^{-tL}\}_{t\geq 0}$ and
let
\begin{align}\label{eq mu}
\mu_0:=\min\lf\{\az_0,2-\frac{n}{q}\r\},
\end{align}
where $\az_0\in(0,1]$ is as in Lemma \ref{lem 1-1}.
\begin{enumerate}
\item[(i)] For any given $k,\,N\in\nn$,
there exist positive constants $C_{(N)}$ and $c$ such that,
for any $t\in(0,\,\fz)$ and every $(x,y)\in\rn\times\rn$,
\begin{equation}\label{eq Gauss1}
0\le \lf|\frac{\partial}{\partial t^k}K_t(x,\,y)\r|
\le\frac{C_{(k,\,N)}}{t^{k+n/2}}e^{-c\frac{|x-y|^2}{t}}
\lf[1+\sqrt{t}m(x,\,V)+\sqrt{t}m(y,\,V)\r]^{-N}.
\end{equation}

\item[(ii)] For any given $k,\,N\in\nn$ and $\mu\in(0,\,\mu_0)$,
there exist positive constants $C_{(k,\,N,\,\mu)}$ and $c$ such that,
for any $t\in(0,\,\fz)$ and every $x,y,h\in\rn$ with $|h|\le\sqrt{t}$,
\begin{align}\label{eq Gauss2}
&\lf|\frac{\partial}{\partial t^k}K_t(x+h,\,y)-\frac{d}{dt^k}K_t(x,\,y)\r|
+\lf|\frac{\partial}{\partial t^k}K_t(x,\,y+h)-\frac{d}{dt^k}K_t(x,\,y)\r|\\
&\hs\le \frac{C_{(k,\,N,\,\mu)}}{t^{k+n/2}}\lf[\frac{|h|}{\sqrt{t}}\r]^\mu
e^{-c\frac{|x-y|^2}{t}}
\lf[1+\sqrt{t}m(x,\,V)+\sqrt{t}m(y,\,V)\r]^{-N}.\noz
\end{align}
\end{enumerate}
\end{lemma}

\begin{remark}\label{rem-2}
By Lemma \ref{lem 1-2}(i), we know that
the heat kernel $K_t$ satisfies the Gaussian upper bound.
From this and \cite[(3.2)]{y08}, we deduce that, for any $p\in (1,\,\fz)$,
the quadratic operator $S_L$ (see \eqref{eq quadratic})
is bounded on $L^p(\rn)$.
\end{remark}

The \emph{Hardy-Littlewood maximal operator} $\cm$ is defined by setting, for any $f\in L^1_{\rm loc}(\rn)$
and $x\in\rn$,
\begin{equation}\label{eq h-l}
\cm(f)(x):=\sup_{B\ni x}\frac{1}{|B|}\int_B |f(y)|\,dy,
\end{equation}
where the supremum is taken over all balls $B$ of $\rn$ containing $x$.

The following lemma establishes the boundedness of $\cm$ on $\vp$,
which is just \cite[Theorem 4.3.8]{dhhr11} (see also \cite[Theorem 3.16]{cf13}).
\begin{lemma}[\cite{dhhr11}]\label{lem 1-3}
Let $p(\cdot)\in C^{\log}(\rn)$ and $1<p_-\le p_+<\fz$. Then there exists a positive
constant $C$ such that, for any $f\in\vp$,
$$\|\cm(f)\|_{\vp}\le C\|f\|_{\vp}.$$
\end{lemma}

The following Fefferman-Stein vector-valued inequality of $\cm$ on $\vp$ is
proved in \cite[Corollary 2.1]{cfmp06}.
\begin{lemma}[\cite{cfmp06}]\label{lem fs}
Let $q\in(1,\,\fz)$ and $p(\cdot)\in C^{\log}(\rn)$ with $p_-\in(1,\,\fz)$.
Then there exists a positive constant $C$ such that, for any sequence $\{f_j\}_{j\in\nn}$
of measurable functions,
\begin{align*}
\lf\|\lf\{\sum_{j=1}^\fz[\cm(f_j)]^q\r\}^{\frac1q}\r\|_{\vp}
\le C\lf\|\lf(\sum_{j=1}^\fz|f_j|^q\r)^{\frac1q}\r\|_{\vp}.
\end{align*}
\end{lemma}

The following lemma is a particular case of \cite[Lemma 2.4]{yz17},
which is is a slight variant of \cite[Lemma 4.1]{s13}.
\begin{lemma}[\cite{yz17}]\label{lem-key}
Let $\kappa\in[1,\,\fz)$,
$p(\cdot)\in C^{\log}(\rn)$, $\underline{p}:=\min\{p_-,\,1\}$ and
$r\in[1,\,\fz]\cap(p_+,\,\fz]$, where $p_-$ and $p_+$ are as in \eqref{eq var}.
Then there exists a positive constant $C$ such that, for any sequence $\{B_j\}_{j\in\nn}$
of balls in $\rn$, $\{\lz_j\}_{j\in\nn}\st\mathbb{C}$ and functions $\{a_j\}_{j\in\nn}$ satisfying
that, for any $j\in\nn$, $\supp a_j\st \kappa B_j$ and $\|a_j\|_{L^r(\rn)}\le |B_j|^{1/r}$,
\begin{align}\label{eq-key}
\lf\|\lf(\sum_{j=1}^\fz|\lz_j a_j|^{\uz{p}}\r)^{\frac{1}{\uz{p}}}\r\|_{\vp}
\le C\kappa^{n(\frac1{\uz{p}}-\frac1r)}\lf\|\lf(\sum_{j=1}^\fz|\lz_j
\chi_{B_j}|^{\uz{p}}\r)^{\frac{1}{\uz{p}}}\r\|_{\vp}.
\end{align}
\end{lemma}

For more properties of the variable Lebesgue spaces $\vp$, we refer the
reader to \cite{cf13,dhhr11}.
\begin{remark}\label{rem-1}
Let $p(\cdot)\in\cp(\rn)$.
\begin{enumerate}
\item[(i)] For any $\lz\in\mathbb{C}$ and $f\in\vp$, $\|\lz f\|_{\vp}=|\lz|\|f\|_{\vp}$.
In particular, if $p_-\in[1,\,\fz)$, then $\|\cdot\|_{\vp}$ is a norm, namely,
for any $f,\,g\in\vp$,
$$\|f+g\|_{\vp}\le \|f\|_{\vp}+\|g\|_{\vp}.$$

\item[(ii)] By the definition of $\|\cdot\|_{\vp}$ (see \eqref{eq norm}),
it is easy to see that, for any $f\in\vp$ and $s\in(0,\,\fz)$,
\begin{equation*}
\lf\||f|^s\r\|_{\vp}=\|f\|_{L^{sp(\cdot)}(\rn)}^s.
\end{equation*}
\end{enumerate}
\end{remark}

\section{Proofs of Theorem \ref{thm-2} and Corollary \ref{cor-1}}\label{s3}
In this section, we give the proofs of Theorem \ref{thm-2} and Corollary \ref{cor-1}.
\subsection{Proof of Theorem \ref{thm-2}}\label{s3-1}
We first prove the following Propositions \ref{pro-2} and \ref{pro-3} below.
Then Theorem \ref{thm-2} is a direct consequence of them.

The following proposition shows that $\ath\st\vhp$.
\begin{proposition}\label{pro-2}
Let $L$ be as in \eqref{eq o} with $A$ satisfying Assumption \ref{as 1} and $r\in (1,\,\fz)$.
Assume that $p(\cdot)\in C^{\log}(\rn)$ with $\frac{n}{n+\mu_0}<p_-\le p_+\le 1$,
where $\mu_0\in(0,1]$ is as in \eqref{eq mu}.
Then there exists a positive constant $C$ such that,
for any $f\in\ath$,
$$\|f\|_{\vhp}\le C\|f\|_{\ath}.$$
\end{proposition}

\begin{proof}
By a density argument, we only need to show that,
for any $f\in\atom$ (see Definition \ref{def at-Hardy}),
$\|f\|_{\vhp}\ls\|f\|_{\ath}$.

Let $f\in\atom$. Then, by \eqref{eq at-Hardy}, we know that
there exist $\{\lz_j\}_{j\in\nn}\st\mathbb{C}$
and a family $\{a_j\}_{j\in\nn}$ of $(p(\cdot),\,r)$-atoms, associated with balls $\{B_j\}_{j\in\nn}$
of $\rn$, such that
\begin{equation}\label{eq 3.0}
f=\sum_{j=1}^\fz \lz_j a_j\ \ \text{in}\ \ L^2(\rn)
\end{equation}
and
\begin{equation}\label{eq 3.0x}
\|f\|_{\ath}\sim\ca(\{\lz_j\}_{j\in\nn},\,\{B_j\}_{j\in\nn}).
\end{equation}
Next, we prove
\begin{align}\label{eq 2.x}
\|S_L(f)\|_{\vp}\ls\ca(\{\lz_j\}_{j\in\nn},\,\{B_j\}_{j\in\nn}).
\end{align}
Indeed, By \eqref{eq 3.0} and the fact that
$S_L$ is bounded on $L^2(\rn)$ (see Remark \ref{rem-2}), we find that
\begin{align*}
\lim_{N\to\fz}\lf\|S_L(f)-S_L\lf(\sum_{j=1}^N\lz_j a_j\r)\r\|_{L^2(\rn)}=0,
\end{align*}
which further implies that, for almost every $x\in\rn$,
\begin{align*}
S_L(f)(x)
&\le \sum_{j=1}^\fz|\lz_j|S_L(a_j)(x)\\
&=\sum_{j=1}^\fz\sum_{i=0}^\fz|\lz_j|S_L(a_j)(x)\chi_{U_i(B_j)}(x),
\end{align*}
where $U_i(B_j)$ is as in \eqref{eq-ujb} with $B$ therein replaced by $B_j$.
From this, Remark \ref{rem-1} and the fact that $p_-\in(0,\,1]$, it follows that
\begin{align}\label{eq 3.3}
\|S_L(f)\|_{\vp}^{p_-}
&=\lf\|[S_L(f)]^{p_-}\r\|_{L^{\frac{p(\cdot)}{p_-}}(\rn)}\noz\\
&\le \sum_{i=0}^\fz\lf\|\sum_{j=1}^\fz|\lz_j|^{p_-}
[S_L(m_j)\chi_{U_i(B_j)}]^{p_-}\r\|_{L^{\frac{p(\cdot)}{p_-}}(\rn)}\noz\\
&= \sum_{i=0}^\fz\lf\|\lf\{\sum_{j=1}^\fz|\lz_j|^{p_-}
[S_L(m_j)\chi_{U_i(B_j)}]^{p_-}\r\}^{\frac{1}{p_-}}\r\|^{p_-}_{\vp}.
\end{align}
To prove \eqref{eq 2.x}, we claim that
it is enough to show that
there exist constants $C\in(0,\,\fz)$
and $\theta\in(n[\frac{1}{p_-}-\frac1r],\fz)$ such that, for any $(p(\cdot),\,r)$-atom $a$,
associated with ball $B:=B(x_B,r_B)$ of $\rn$, and any $i\in\zz_+$,
\begin{equation}\label{eq 3.2}
\|S_L(a)\|_{L^r(U_i(B))}\le C 2^{-i\theta}|B|^{\frac1r}\|\chi_B\|^{-1}_{\vp}.
\end{equation}
Indeed, if \eqref{eq 3.2} holds true, then, for any $i\in\zz_+$ and $j\in\nn$,
\begin{align*}
\lf\|2^{i\tz}\lf\|\chi_{B_j}\r\|_{\vp}S_L(a_j)\chi_{U_i(B_j)}\r\|_{L^r(\rn)}
\ls |B_j|^{1/r}.
\end{align*}
By this and Lemma \ref{lem-key}, we know that
\begin{align*}
&\lf\|\lf\{\sum_{j=1}^\fz \lf[|\lz_j|S_L(m_j)\chi_{U_i(B_j)}\r]^{p_-}\r\}^{\frac{1}{p_-}}\r\|_{\vp}\\
&\hs\ls2^{in(\frac1{p_-}-\frac1r)}\lf\|\lf\{\sum_{j=1}^\fz
\lf[2^{-i\tz}\|\chi_{B_j}\|_{\vp}^{-1}|\lz_j|\chi_{B_j}\r]^{p_-}\r\}^
{\frac{1}{p_-}}\r\|_{\vp}\noz\\
&\hs\ls2^{-i[\tz-n(\frac1{p_-}-\frac1r)]}
\mathcal{A}(\{\lz_j\}_{j\in\nn},\,\{B_j\}_{j\in\nn}).\noz
\end{align*}
Combining this, \eqref{eq 3.3} and the fact that
$\tz>n(\frac1{p_-}-\frac1r)$, we further conclude that, for any $f\in\atom$,
\begin{align*}
\|S_L(f)\|_{\vp}
&\ls\lf\{\sum_{i=0}^\fz 2^{-i[\tz-n(\frac1{p_-}-\frac1r)]p_-}\r\}^{\frac1{p_-}}
\mathcal{A}(\{\lz_j\}_{j\in\nn},\,\{B_j\}_{j\in\nn})\\
&\ls\|f\|_{\ath},
\end{align*}
which is the desired result.

Next, we prove \eqref{eq 3.2}.
When $i\in\{0,\ldots,10\}$, by the boundedness of
$S_L$ on $L^r(\rn)$ (see Remark \ref{rem-2}), we obtain
\begin{align*}
\lf\|S_L(a)\r\|_{L^r(U_i(B))}\ls\|a\|_{L^r(\rn)}\ls|B|^{\frac1r}\lf\|\chi_B\r\|^{-1}_{\vp}.
\end{align*}
When $i\geq 11$ and $i\in\zz_+$, for any given $\eta\in(0,1)$,
we have
\begin{align*}
\lf\|S_L(a)\r\|_{L^r(U_i(B))}
&=\lf\{\int_{U_i(B)}\lf[\int_0^\fz\int_{B(x,\,t)}\lf|t^2Le^{-t^2L}(a)(y)\r|^2\,
\dydt\r]^{\frac r2}\,dx\r\}^{\frac1r}\\
&\ls\lf\{\int_{U_i(B)}\lf[\int_0^{2^{i\eta} r_B}\int_{B(x,\,t)}\lf|t^2Le^{-t^2L}(a)(y)\r|^2\,
\dydt\r]^{\frac r2}\,dx\r\}^{\frac1r}\\
&\hs+\lf\{\int_{U_i(B)}\lf[\int_{2^{i\eta} r_B}^\fz\cdots\r]^{\frac r2}\,dx\r\}^{\frac1r}\\
&=:{\rm I}+{\rm II}.
\end{align*}
For ${\rm I}$, by the fact that
$-Le^{-tL}(a)(y)=\frac{\partial}{\partial t}e^{-tL}(a)(y)
=\int_\rn\frac{\partial}{\partial t}K_t(y,z)\,a(z)\,dz$ and Lemma \ref{lem 1-2}(i),
we find that
\begin{align}\label{eq 2.0}
\lf|t^2L e^{-t^2L}(a)(y)\r|
\ls\int_\rn\frac{1}{t^n}e^{-c\frac{|y-z|^2}{t^2}}|a(z)|\,dz.
\end{align}
Moreover, from fact that $z\in B(x_B,r_B)$,
$y\in B(x,\,t)$, $x\in U_i(B)$ and $t\in(0,2^{i\eta}r_B)$, it follows that
\begin{align*}
|y-z|\geq |x-x_B|-|y-x|-|z-x_B|
\geq 2^{i-1}r_B-2^{i\eta}r_B-r_B
\gtrsim 2^ir_B.
\end{align*}
This, combined with \eqref{eq 2.0}, the H\"{o}lder inequality and
$a$ is an $(p(\cdot),\,r)$-atom associated with ball $B(x_B,r_B)$, implies that
\begin{align*}
\lf|t^2L e^{-t^2L}(a)(y)\r|
\ls\frac{1}{t^n}e^{-c\frac{(2^{i}r_B)^2}{t^2}}
\|a\|_{L^r(B)}|B|^{1-1/r}
\ls\frac{1}{t^n}e^{-c\frac{(2^{i}r_B)^2}{t^2}}|B|\|\chi_B\|_{\vp}^{-1}.
\end{align*}
Hence, we obtain
\begin{align}\label{eq 2.2}
{\rm I}
&\ls\lf\{\int_{U_i(B)}\lf[\int_0^{2^{i\eta} r_B}\int_{B(x,\,t)}
\frac{1}{t^{2n}}e^{-c\frac{(2^{i}r_B)^2}{t^2}}|B|^2\|\chi_B\|_{\vp}^{-2}\,
\dydt\r]^{\frac r2}\,dx\r\}^{\frac1r}\noz\\
&\ls\lf\{\int_{U_i(B)}\lf[\int_0^{2^{i\eta} r_B}
\frac{1}{t^{2n+1}}\lf(\frac{t}{2^ir_B}\r)^{2N}|B|^2\|\chi_B\|_{\vp}^{-2}\,
dt\r]^{\frac r2}\,dx\r\}^{\frac1r}\noz\\
&\ls\lf[\int_{U_i(B)}2^{-irN}2^{i\eta r(N-n)}r_B^{-rn}\,dx\r]^{\frac1r}
|B|\|\chi_B\|^{-1}_{\vp}\noz\\
&\ls 2^{-i[N(1-\eta)+n(\eta-1/r)]}|B|^{1/r}\|\chi_B\|_{\vp}^{-1}.
\end{align}

To estimate ${\rm II}$,
we consider two cases. In what follows, we denote the kernel of $tLe^{-tL}$
by $V_t(\cdot,\cdot)$, namely, $V_t(\cdot,\cdot):=t\frac{\partial}{\partial t}K_t(\cdot,\cdot)$.

Case i): $r_B<[m(x_B,V)]^{-1}$. In this case,
by Definition \ref{def atom}(iii), we know that $\int_\rn a(x)\,dx=0$.
By this, Lemma \ref{lem 1-2}(ii),
Definition \ref{def atom} and the H\"{o}lder inequality,
we find that, for any $y\in B(x,\,t)$,
\begin{align*}
\lf|t^2Le^{-t^2L}(a)(y)\r|
&=\lf|\int_\rn V_{t^2}(y,z)a(z)\,dz\r|\\
&=\lf|\int_\rn\lf[V_{t^2}(y,z)-V_{t^2}(y,x_B)\r]a(z)\,dz\r|\\
&\ls\int_\rn\frac{1}{t^n}\lf(\frac{|z-x_B|}{t}\r)^\mu
e^{-c\frac{|y-x_B|^2}{t^2}}|a(z)|\,dz\\
&\ls\frac{r_B^\mu}{t^{n+\mu}}\|\chi_B\|^{-1}_{\vp}|B|,
\end{align*}
where $\mu\in(0,\,\mu_0)$ is determined later
and $\mu_0\in(0,\,1]$ is as in \eqref{eq mu}.
Therefore, we have
\begin{align}\label{eq 2.3}
{\rm II}
&\ls\lf\{\int_{U_i(B)}\lf[\int_{2^{i\eta}r_B}^\fz\int_{B(x,\,t)}
\frac{r_B^{2\mu}}{t^{2(n+\mu)}}|B|^2\|\chi_B\|_{\vp}^{-2}\,\dydt\r]^{r/2}
\,dx\r\}^{1/r}\noz\\
&\ls\lf\{\int_{U_i(B)}\lf[\int_{2^{i\eta}r_B}^\fz
\frac{r_B^{2\mu}}{t^{2(n+\mu)+1}}\,dt\r]^{r/2}\,dx\r\}^{1/r}|B|\|\chi_B\|^{-1}_{\vp}\noz\\
&\ls\lf\{\int_{U_i(B)}2^{-iq\eta(n+\mu)}r_B^{-rn}\,dx\r\}^{1/r}|B|\|\chi_B\|^{-1}_{\vp}\noz\\
&\ls 2^{-i[(\eta-1/r)n+\eta\mu]}|B|^{1/r}\|\chi_B\|_{\vp}^{-1}.
\end{align}

Case ii): $r_B\geq [m(x_B,V)]^{-1}$.
In this case, by Lemma \ref{lem 1-2}(i), we find that, for any $y\in B(x,\,t)$,
\begin{align}\label{eq 2.1}
\lf|t^2Le^{-t^2L}(a)(y)\r|
&=\lf|\int_\rn V_{t^2}(y,z)a(z)\,dz\r|\noz\\
&\ls\int_\rn\frac{1}{t^n}e^{-c\frac{|y-z|^2}{t^2}}\lf[1+tm(y,\,V)+tm(z,V)\r]^{-N}|a(z)|\,dz,
\end{align}
where $N\in(0,\,\fz)$ is determined later.
Moreover, by the fact that $y\in B(x,\,t)$, $t\in(2^{i\eta}r_B,\fz)$,
$z\in B(x_B,r_B)$, $r_B m(x_B,V)\geq 1$ and Lemma \ref{lem aux}(iii),
we know that there exists a positive constant $k_0$ such that
\begin{align*}
tm(z,V)
&\gtrsim tm(x_B,V)[1+|z-x_B|m(x_B,V)]^{-\frac{k_0}{k_0+1}}\\
&\gtrsim 2^{i\eta}r_Bm(x_B,V)[1+r_Bm(x_B,V)]^{-\frac{k_0}{k_0+1}}
\gtrsim 2^{i\eta}[r_Bm(x_B,V)]^{\frac1{k_0+1}}
\gtrsim 2^{i\eta}.
\end{align*}
From this, \eqref{eq 2.1}, the H\"{o}lder inequality
and Definition \ref{def atom}(ii), we deduce that
\begin{align*}
\lf|t^2Le^{-t^2L}(a)(y)\r|
\ls\int_\rn 2^{-i\eta N}\frac{1}{t^n}|a(z)|\,dz
\ls 2^{-i\eta N}\frac{1}{t^n}|B|\|\chi_B\|_{\vp}^{-1}.
\end{align*}
This further implies that
\begin{align}\label{eq 2.4}
{\rm II}
&\ls\lf\{\int_{U_i(B)}\lf[\int_{2^{i\eta}r_B}^\fz\int_{B(x,\,t)}
\frac{2^{-2i\eta N}}{t^{2n}}|B|^2\|\chi_B\|_{\vp}^{-2}\,\dydt\r]^{r/2}
\,dx\r\}^{1/r}\noz\\
&\ls\lf\{\int_{U_i(B)}\lf[\int_{2^{i\eta}r_B}^\fz
\frac{2^{-2i\eta N}}{t^{2n+1}}\,dt\r]^{r/2}\,dx\r\}^{1/r}|B|\|\chi_B\|^{-1}_{\vp}\noz\\
&\ls 2^{-i[\eta(N+n)-n/r]}|B|^{1/r}\|\chi_B\|_{\vp}^{-1}.
\end{align}

Let $\theta:=\min\{N(1-\eta)+n(\eta-1/r), (\eta-1/r)n+\eta\mu, \eta(N+n)-n/r\}$.
Then, by choosing $N$ large enough, $\eta$ close enough to $1$, $\mu$ close enough to $\mu_0$
and the fact that $p_->\frac{n}{n+\mu_0}$, we know that
$\theta\in(n[\frac{1}{p_-}-\frac1r],\fz)$.
This, combined with \eqref{eq 2.2}, \eqref{eq 2.3} and \eqref{eq 2.4},
implies \eqref{eq 3.2}, which completes the proof of
Proposition \ref{pro-2}.
\end{proof}

Next, by means of the atomic decomposition of variable tent space,
we show that $\vhp\st\ath$.

The following variable tent space is first
introduced by Zhuo et al. \cite[p.\,1567]{zyl14}.
\begin{definition}[\cite{zyl14}]\label{def tent}
Let $p(\cdot)\in\cp(\rn)$. The \emph{variable tent space $T^{p(\cdot)}(\rnn)$}
is defined to be the space of all measurable functions $f$ such that
$\|f\|_{T^{p(\cdot)}(\rnn)}:=\|A(f)\|_{\vp}<\fz$,
where, for any $x\in\rn$,
\begin{equation*}
A(f)(x):=\lf[\int_0^\fz\int_{B(x,\,t)}|f(y,\,t)|^2\dydt\r]^{1/2}.
\end{equation*}
\end{definition}

\begin{remark}\label{rem-4}
In particular, if $p(\cdot)\equiv p\in(0,\,\fz)$ is a constant exponent,
$T^{p(\cdot)}(\rnn)$ is just the tent space $T^{p}(\rnn)$ introduced by
Coifman et al. \cite{cms85}. Moreover, by \cite[Theorem 2]{cms85},
we know that if $p\in(1,\,\fz)$, then, for any $f\in T^p(\rnn)$ and $g\in T^{p'}(\rnn)$,
the pairing
\begin{align*}
\langle f,\,g\rangle:=\int_0^\fz\int_{\cx} f(x,\,t)\ov{g(x,\,t)}\,d\mu(x)\,\frac{dt}{t}
\end{align*}
realizes $T^{p'}(\rnn)$ as the dual of $T^p(\rnn)$, up to equivalent norms,
where $1/p+1/p'=1$.
\end{remark}

For any open set $O\st\rn$, the \emph{tent} over $O$  is defined by setting
\begin{align*}
\wh{O}:=\lf\{(y,\,t)\in\rnn:\ \dist \lf(y,\,O^\complement\r)\geq t\r\}.
\end{align*}

Let $p(\cdot)\in\cp(\rn)$. Recall that a measurable
function $A$ on $\rnn$ is called a \emph{$(p(\cdot),\,\fz)_T$-atom} if there
exists a ball $B\st\rn$ such that
\begin{enumerate}
\item[(i)] $\supp A\st\wh{B}$;

\item[(ii)] for any $r\in(1,\,\fz)$,
$\|A\|_{T^r(\rnn)}\le |B|^{1/r}\|\chi_B\|_{\vp}^{-1}.$
\end{enumerate}

The following establishes the atomic decomposition of $T^{p(\cdot)}(\rnn)$,
which is just \cite[Lemma 3.3]{yzz15} (see also \cite[Theorem 2.16]{zyl14}).
\begin{lemma}[\cite{yzz15}]\label{lem 2.0}
Let $p(\cdot)\in C^{\log}(\rn)$ with $p_+\in(0,\,1]$. Then, for any $f\in T^{p(\cdot)}(\rnn)$,
there exist $\{\lz_j\}_{j\in\nn}\st\cc$ and a family $\{A_j\}_{j\in\nn}$ of $(p(\cdot),\,\fz)_T$-atoms
such that, for almost every $(x,\,t)\in\rnn$,
\begin{equation*}
f(x,\,t)=\sum_{j\in\nn}\lz_j A_j(x,\,t)
\end{equation*}
and
$$\ca(\{\lz_j\}_{j\in\nn},\,\{B_j\}_{j\in\nn})
\le C\|f\|_{T^{p(\cdot)}(\rnn)},$$
where, for any $j\in\nn$, $B_j$ is the ball associated with $a_j$
and $C$ a positive constant independent of $f$.
\end{lemma}

By Lemma \ref{lem 2.0}, we obtain the following proposition.
\begin{proposition}\label{pro-3}
Let $L$ be as in \eqref{eq o} with $A$ satisfying Assumption \ref{as 1},
$\mu_0\in(0,\,1]$ as in \eqref{eq mu},
$p(\cdot)\in C^{\log}(\rn)$ with $\frac{n}{n+\mu_0}<p_-\le p_+\le 1$
and $r\in(1,\,\fz)$ such that $\mu_0+\frac{n}{r}>\frac{n}{p_-}$.
Then there exists a positive constant $C$ such that,
for any $f\in H_L^{p(\cdot)}(\rn)$,
\begin{equation}\label{eq 2.7z}
\|f\|_{\ath}\le C\|f\|_{H_L^{p(\cdot)}(\rn)}.
\end{equation}
\end{proposition}

\begin{proof}
By a density argument,
we only need to show \eqref{eq 2.7z} holds true for any $f\in H_L^{p(\cdot)}(\rn)\cap L^2(\rn)$.
Let $f\in H_L^{p(\cdot)}(\rn)\cap L^2(\rn)$.
Then, for any $x\in\rn$ and $t\in (0,\,\fz)$, define
$F(x,t):=t^2Le^{-t^2L}(f)(x)$.
Since $t^2Le^{-t^2L}$ is bounded on $L^2(\rn)$,
we conclude that $F\in T^{p(\cdot)}(\rr^{n+1}_+)\cap T^2(\rnn)$.
Then, by Lemma \ref{lem 2.0}, we further find that
there exist $\{\lz_j\}_{j\in\nn}\st\cc$ and
a family $\{A_j\}_{j\in\nn}$ of $(p(\cdot),\,\fz)$-atoms,
associated with balls $\{B_j\}_{j\in\nn}$, such that
\begin{align}\label{eq 2.7}
t^2Le^{-t^2L}(f)(x)=\sum_{j=1}^\fz\lz_jA_j\quad\text{in}\quad T^{p(\cdot)}(\rnn)\cap T^2(\rnn)
\end{align}
and
\begin{align}\label{eq 2.7x}
\ca(\{\lz_j\}_{j\in\nn},\,\{B_j\}_{j\in\nn})
\sim \|F\|_{T^{p(\cdot)}(\rnn)}\sim\|f\|_{H_L^{p(\cdot)}(\rn)}.
\end{align}

Let $\phi\in C_c^\fz(\rr)$ be an odd function which satisfies
$\supp\phi\st(-1,1)$ and $\int_\rr\phi(x)\,dx=0$,
where $C_c^\fz(\rr)$ denotes the set of all finitely differential function $f$
on $\rr$ with compact support.
For any $z\in\cc$, define
$
\Phi(z):=\int_\rr e^{-2\pi izt}\phi(t)\,dt.
$
Then it is easy to see that $\Phi$ is a holomorphic function on $\cc$.
Moreover, for any given $\mu\in(0,\,\pi/2)$, there exists a positive constant
$C:=C_{(\mu,\,\phi)}$ such that, for any
$z\in\{z\in\cc:\ |\arg z|<\mu\}$,
\begin{align}\label{eq 2.5}
|\Phi(z)|\le C\frac{|z|}{1+|z|^2}.
\end{align}
Indeed, when $|z|\le 1$, by the fact that $\int_\rr\phi(t)\,dt=0$,
we find that
\begin{align}\label{eq 2.6}
|\Phi(z)|
=\lf|\int_\rr\lf(e^{-2\pi itz}-1\r)\phi(t)\,dt\r|
\ls2\pi|z|\int_\rr|t\phi(t)|\,dt\ls|z|.
\end{align}
when $|z|>1$, by integration by parts, we have
\begin{align*}
|-2\pi i z\Phi(z)|
=\lf|\int_\rr -2\pi ize^{-2\pi izt}\phi(t)\,dt\r|
=\lf|\int_\rr e^{-2\pi izt}\frac{d}{dt}\phi(t)\,dt\r|\ls 1,
\end{align*}
namely, $|\Phi(z)|\ls |z|^{-1}$. This, combined with \eqref{eq 2.6},
implies \eqref{eq 2.5}.
Since $L$ is a nonnegative self-adjoint operator on $L^2(\rn)$,
by \cite[Lecture 2]{adm96}, we know that $L$ has a
holomorphic functional calculus.
By this and \eqref{eq 2.5}, for any $t\in(0,\,\fz)$, we could define
\begin{align*}
\Phi\lf(t\sqrt{L}\r)=\int_{\gamma_\pm}\Phi\lf(t\sqrt{\zeta}\r)(\zeta I-L)^{-1}\,d\zeta,
\end{align*}
where $\gamma_\pm:=\{re^{i\nu}:\ r\in(0,\,\fz)\}\cup\{re^{-i\nu}:\ r\in(0,\,\fz)\}$,
for any given $\nu\in(0,\,\pi/2)$, is a curve consisting of two rays parameterized anti-clockwise.
From this and the Laplace identity (see \cite[p.\,484]{ka95}), we further deduce that
\begin{align}\label{eq 2.17}
\Phi\lf(t\sqrt{L}\r)
&=\int_{\gamma_\pm}\Phi\lf(t\sqrt{\zeta}\r)(\zeta I-L)^{-1}\,d\zeta\noz\\
&=\int_{\gamma_\pm}\Phi\lf(t\sqrt{\zeta}\r)\int_{\{z=re^{\pm i\frac25\pi}:\ r\in(0,\,\fz)\}}
e^{\zeta z}e^{-zL}\,dz\,d\zeta.
\end{align}

For any $G\in T^2(\rnn)$, define
\begin{align*}
\pi_L(G)(x):=\int_0^\fz t^2L\Phi\lf(t\sqrt{L}\r)(G(\cdot,t))(x)\,\frac{dt}{t}.
\end{align*}
Then, from the bounded functional calculus of $L$, it follows that,
\begin{align*}
f=c_{(\Phi)}\int_0^\fz t^2L\Phi\lf(t\sqrt{L}\r)\lf(t^2L e^{-t^2L}(f)\r)\,\frac{dt}{t}
\quad\text{in}\quad L^2(\rn),
\end{align*}
where $c_{(\Phi)}$ is a positive constant such that
$$c_{(\Phi)}\int_0^\fz t^4\Phi(t)e^{-t^2}\,\frac{dt}{t}=1.$$
By the fact that $\pi_L$ is bounded from $T^2(\rnn)$ to $L^2(\rn)$ (see \cite[Lemma 3.4]{y08})
and \eqref{eq 2.7}, we find that
\begin{align}\label{eq 2.8}
f=c_{(\Phi)}\pi_L\lf(\sum_{j=1}^\fz\lz_j A_j\r)
=c_{(\Phi)}\sum_{j=1}^\fz\lz_j\pi_L(A_j)\quad\text{in}\quad L^2(\rn).
\end{align}

Next, we show that \eqref{eq 2.8} can be rewritten as
an atomic $(p(\cdot),\,r)$-representation of $f$ (see Definition \ref{def at-Hardy})
for any $r\in(1,\,\fz)$.
Indeed, for any
$(p(\cdot),\,\fz)_T$-atom $A$, associated with ball $B:=B(x_B,r_B)$,
define
\begin{align*}
a:=\pi_L(A)=\int_0^\fz t^2L\Phi\lf(t\sqrt{L}\r)(A(\cdot,t))\,\frac{dt}{t}.
\end{align*}
Then, by the finite speed propagation property of $L$ (see \cite[Lemma 3.5]{hlmmy11}),
we conclude that $\supp a\st 2B$.
Now we consider two cases.

Case i): $[m(x_B,V)]^{-1}\le 2r_B<\fz$.
In this case, by the fact that, for any $r\in(1,\,\fz)$,
$\pi_L$ is bounded from $T^r(\rnn)$ to $L^r(\rn)$ (see \cite[Lemma 3.4]{y08}),
we have
\begin{align*}
\|a\|_{L^r(\rn)}\ls\|A\|_{T^r(\rnn)}\ls|B|^{1/r}\|\chi_B\|_{\vp}^{-1}.
\end{align*}

Case ii): $0<2r_B<[m(x_B,V)]^{-1}$. In this case,
we know that there exists a unique $i_0\in\nn$ such that
$2^{i_0}r_B\le[m(x_B,V)]^{-1}<2^{i_0+1}r_B$.
Write
\begin{align}\label{eq 2.9x}
a(x)
&:=\lf[a(x)-\frac{1}{|2B|}\int_{2B}a(y)\,dy\,\chi_{2B}(x)\r]\noz\\
&\hs+\lf[\frac{1}{|2B|}\int_{2B}a(y)\,dy\,\chi_{2B}(x)-
\frac{1}{|2^{i_0+1}B|}\int_{2^{i_0+1}B}a(y)\,dy\,\chi_{2^{i_0+1}B}(x)\r]\noz\\
&\hs+\lf[\frac{1}{|2^{i_0+1}B|}\int_{2^{i_0+1}B}a(y)\,dy\,\chi_{2^{i_0+1}B}(x)\r]\noz\\
&=:{\rm I}(x)+{\rm II}(x)+{\rm III}(x).
\end{align}

For ${\rm I}$, it is easy to see that $\supp{\rm I}\st 2B$, $\int_{2B}{\rm I}(y)\,dy=0$
and
\begin{align*}
\|{\rm I}\|_{L^r(\rn)}
&\le\|a\|_{L^r(\rn)}+|2B|^{\frac1r-1}\int_{2B}|a(y)|\,dy\\
&\le 2\|a\|_{L^r(\rn)}\ls|B|^{1/r}\|\chi_B\|_{\vp}^{-1}.
\end{align*}

For ${\rm II}$, by the fact that $\supp a\st 2B$ and \eqref{eq 2.17}, we find that
\begin{align}\label{eq 2.9y}
\int_{2B}a(y)\,dy
&=\int_{2^{i_0}B}a(y)\,dy
=\int_{2^{i_0}B}\lf[\int_0^\fz t^2L\Phi\lf(t\sqrt{L}\r)(A(\cdot,t))(x)\,\frac{dt}{t}\r]\,dx\noz\\
&=\int_{2^{i_0}B}\lf[\int_0^\fz t^2\int_{\gamma_{\pm}}\Phi\lf(t\sqrt{\zeta}\r)\r.\noz\\
&\hs\times\lf.\int_{\{z=re^{\pm i\frac25\pi}:\ r\in(0,\,\fz)\}}e^{\zeta z}
Le^{-zL}(A(\cdot,t))(x)\,dz\,d\zeta\,\frac{dt}{t}\r]\,dx\noz\\
&=\int_{2^{i_0}B}\lf[\int_0^\fz t^2\int_{\gamma_{\pm}}\Phi\lf(t\sqrt{\zeta}\r)
\int_{\{z=re^{\pm i\frac25\pi}:\ r\in(0,\,\fz)\}}e^{\zeta z}\r.\noz\\
&\hs\hs\lf.\times\int_B\frac{\partial}{\partial z}K_z(x,\,y)A(y,t)\,dy\,dz\,
d\zeta\,\frac{dt}{t}\r]\,dx.
\end{align}
For any given $\delta\in(0,\,\mu_0)$ with
$\mu_0$ as in \eqref{eq mu}, we claim that there exists a positive
constant $C$ such that,
for any $z\in\Sigma^0_{\pi/5}$ and $y\in B(x_B,r_B)$,
\begin{align}\label{eq 2.16}
\lf|\int_{B(x_B,2[m(x_B,\,r_B)]^{-1})}\frac{\partial}{\partial z}K_z(x,\,y)\,dx\r|
\le C\frac{1}{|z|}\lf\{\frac{\sqrt{|z|}}{[m(x_B,\,r_B)]^{-1}}\r\}^\delta,
\end{align}
where $\Sigma^0_{\pi/5}$ is as in \eqref{eq 0.1}.
The proof of \eqref{eq 2.16} is totally similar to that of
\cite[(2.22)]{ccyy14} (see also the proof of \cite[Lemma A.5(b)]{ar03}),
the details being omitted.
From \eqref{eq 2.9y}, the Fubini theorem,
the fact that $\supp A\st\widehat{B}$, \eqref{eq 2.16}
and Remark \ref{rem-4}, we deduce that
\begin{align}\label{eq 2.10x}
\lf|\int_{2B}a(x)\,dx\r|
&\le\int_B\int_0^{r_B}t^2\int_{\gamma_\pm}\lf|\Phi\lf(t\sqrt{\zeta}\r)\r|
\int_{\{z=re^{\pm i\frac25\pi}:\ r\in(0,\,\fz)\}}\lf|e^{\zeta z}\r|\\
&\hs\times\lf|\int_{2^{i_0}B}\frac{\partial}{\partial z}K_z(x,\,y)\,dx\r||A(y,t)|\,
d|z|\,d|\zeta|\,\frac{dt}{t}\,dy\noz\\
&\ls\int_B\int_0^{r_B}t^2\int_{\gamma_\pm}\lf|\Phi\lf(t\sqrt{\zeta}\r)\r|
\int_{\{z=re^{\pm i\frac25\pi}:\ r\in(0,\,\fz)\}}\lf|e^{\zeta z}\r|\noz\\
&\hs\times\frac{1}{|z|}\lf(\frac{\sqrt{|z|}}{2^{i_0}r_B}\r)^\delta|A(y,t)|\,
d|z|\,d|\zeta|\,\frac{dt}{t}\,dy\noz\\
&\ls 2^{-i_0\delta}\int_B\int_0^{r_B}t^2\int_{\gamma_\pm}\lf|\Phi\lf(t\sqrt{\zeta}\r)\r|
\lf(\sqrt{|\zeta|}r_B\r)^{-\delta}|A(y,t)|\,d|\zeta|\,\frac{dt}{t}\,dy\noz\\
&\ls 2^{-i_0\delta}\int_B\int_0^{r_B}\lf(\frac{t}{r_B}\r)^\delta|A(y,t)|\,\frac{dt}{t}\,dy\noz\\
&\ls 2^{-i_0\delta}\|A\|_{T^2(\rnn)}\lf[\int_\rn\int_0^\fz\int_{B(x,\,t)}
\lf(\frac{t}{r_B}\r)^{2\delta}\chi_{\widehat{B}}(y,t)\,\dydt\,dx\r]^{1/2}\noz\\
&\ls 2^{-i_0\delta}|B|^{1/2}\|\chi_B\|_{\vp}^{-1}
\lf[\int_0^{r_B}\int_B\lf(\frac{t}{r_B}\r)^{2\delta}\,dy\,\frac{dt}{t}\r]^{1/2}\noz\\
&\ls 2^{-i_0\delta}|B|\|\chi_B\|_{\vp}^{-1}.\noz
\end{align}
This, combined with the fact that $\supp {\rm II}\st 2^{i_0+1}B$
and $2^{i_0+1}r_B>[m(x_B,V)]^{-1}$, implies that
\begin{align}\label{eq 2.10}
\|{\rm II}\|_{L^r(\rn)}
&\le\lf|\int_{2B}a(x)\,dx\r|\lf[\frac1{|2B|^{1/r'}}+\frac1{|2^{i_0+1}B|^{1/r'}}\r]\noz\\
&\ls 2^{-i_0\delta}|B|^{1-\frac1{r'}}\|\chi_B\|_{\vp}^{-1}
\ls2^{-i_0\delta}|B|^{1/r}\|\chi_B\|_{\vp}^{-1}.
\end{align}

For ${\rm III}$, by \eqref{eq 2.10x}, we know that
\begin{align}\label{eq 2.11}
\|{\rm III}\|_{L^r(\rn)}
\le\lf|\int_{2B}a(x)\,dx\r|\frac{1}{|2^{i_0+1}B|^{1/r'}}
\ls 2^{-i_0\delta}|B|^{1/r}\|\chi_B\|_{\vp}^{-1}.
\end{align}

From the above argument for ${\rm I}$,
${\rm II}$ and ${\rm III}$, \eqref{eq 2.8} and \eqref{eq 2.9x}, it follows that
\begin{align}\label{eq 2.9}
f&=c_{(\Phi)}\sum_{j=1}^\fz\lz_j\pi_L(A_j)\noz\\
&=c_{(\Phi)}\sum_{j\in\{j\in\nn:\ r_{B_j}\geq\frac12[m(x_{B_j},V)]^{-1}\}}\lz_j\pi_L(A_j)\noz\\
&\hs+c_{(\Phi)}\sum_{j\in\{j\in\nn:\ r_{B_j}<\frac12[m(x_{B_j},V)]^{-1}\}}
\lz_j[{\rm I}_j+\lz_{j,\,k_j}(h_{j,\,2}+h_{j,\,3})],
\end{align}
where $k_j\in\nn$ is the unique natural number such that
$2^{k_j}r_{B_j}\le[m(x_{B_j},V)]^{-1}<2^{k_j+1}r_{B_j}$,
\begin{align*}
\lz_{j,\,k_j}:=2^{-k_j\delta}
\frac{|B_j|^{1/r}\|\chi_{2^{k_j+1}B_j}\|_{\vp}}{|2^{k_j+1}B_j|^{1/r}\|\chi_{B_j}\|_{\vp}},
\end{align*}
\begin{align*}
h_{j,\,2}:=\frac{{\rm II}_j(x)|2^{k_j+1}B_j|^{1/r}\|\chi_{B_j}\|_{\vp}}
{|B_j|^{1/r}\|\chi_{2^{k_j+1}B_j}\|_{\vp}}
\end{align*}
and
\begin{align*}
h_{j,\,3}:=\frac{{\rm III}_j(x)|2^{k_j+1}B_j|^{1/r}\|\chi_{B_j}\|_{\vp}}
{|B_j|^{1/r}\|\chi_{2^{k_j+1}B_j}\|_{\vp}}.
\end{align*}
By \eqref{eq 2.10} and \eqref{eq 2.11}, we know that
$h_{j,\,2}$ and $h_{j,\,3}$ are $(p(\cdot),\,r)$-atoms and
\eqref{eq 2.9} is an atomic $(p(\cdot),\,r)$-representation of $f$,
up to a harmless constant multiple.
Thus, we have
\begin{align*}
\|f\|_{\ath}
&\ls\lf\|\lf\{\sum_{j\in\{j\in\nn:\ r_{B_j}\geq\frac12[m(B_j,V)]^{-1}\}}
\lf[\frac{|\lz_j|\chi_{2B_j}}{\|\chi_{2B_j}\|_\vp}\r]^{p_-}\r\}^{1/{p_-}}\r\|_\vp\\
&\hs+\lf\|\lf\{\sum_{j\in\{j\in\nn:\ r_{B_j}<\frac12[m(B_j,V)]^{-1}\}}
\lf[\frac{|\lz_j|\chi_{2B_j}}{\|\chi_{2B_j}\|_\vp}\r]^{p_-}\r\}^{1/{p_-}}\r\|_\vp\\
&\hs+\lf\|\lf\{\sum_{j\in\{j\in\nn:\ r_{B_j}<\frac12[m(B_j,V)]^{-1}\}}
\lf[\frac{|\lz_j\lz_{j,k_j}|\chi_{2^{k_j+1}B_j}}{\|\chi_{2^{k_j+1}B_j}\|_\vp}
\r]^{p_-}\r\}^{1/{p_-}}\r\|_\vp\\
&\ls\lf\|\lf\{\sum_{j\in\nn}
\lf[\frac{|\lz_j|\chi_{2B_j}}{\|\chi_{2B_j}\|_\vp}\r]^{p_-}\r\}^{1/{p_-}}\r\|_\vp\\
&\hs+\lf\|\lf\{\sum_{j\in\nn}\sum_{k\in\nn}
\lf[\frac{|\lz_j\lz_{j,k}|\chi_{2^{k+1}B_j}}{\|\chi_{2^{k+1}B_j}\|_\vp}
\r]^{p_-}\r\}^{1/{p_-}}\r\|_\vp\\
&=:{\rm J}+{\rm K}.
\end{align*}
We first estimate ${\rm K}$. By the fact that $\chi_{2^{k+1}B_j}(x)\le 2^{kn}\cm(\chi_{B_j})(x)$,
where $\cm$ is the Hardy-Littlewood maximal operator,
we find that, for any $s\in(0,\,p_-)$,
$\chi_{2^{k+1}B_j}(x)\le 2^{kn/s}[\cm(\chi_{B_j})(x)]^{1/s}$.
By the fact that $\mu_0+\frac{n}{r}>\frac{n}{p_-}$,
we could fix some $s\in(0,\,p_-)$ and $\delta\in(0,\,\mu_0)$
such that $\delta+\frac{n}{r}>\frac{n}{s}$.
Thus, from this, Remark \ref{rem-1}(i) and (ii) and Lemma \ref{lem fs},
we deduce that
\begin{align*}
{\rm K}
&\le\lf\|\lf\{\sum_{j\in\nn}\sum_{k\in\nn}\lf[2^{-k(\delta+\frac{n}{r}-\frac{n}{s})}
\frac{|\lz_j|[\cm(\chi_{B_j})]^{1/s}}{\|\chi_{B_j}\|_\vp}\r]^{p_-}\r\}^{1/{p_-}}\r\|_{\vp}\\
&=\lf\|\sum_{k\in\nn}\sum_{j\in\nn}\lf[2^{-k(\delta+\frac{n}{r}-\frac{n}{s})}
\frac{|\lz_j|[\cm(\chi_{B_j})]^{1/s}}{\|\chi_{B_j}\|_\vp}\r]^{p_-}\r\|^{\frac{1}{p_-}}
_{L^{\frac{p(\cdot)}{p_-}}(\rn)}\\
&\le\lf\{\sum_{k\in\nn}2^{-kp_-(\delta+\frac{n}{r}-\frac{n}{s})}
\lf\|\sum_{j\in\nn}\lf[\cm\lf(\frac{|\lz_j|^s}{\|\chi_{B_j}\|^s_{\vp}}\chi_{B_j}\r)
\r]^{\frac{p_-}{s}}\r\|_{L^{\frac{p(\cdot)}{p_-}}(\rn)}\r\}^{1/p_-}\\
&=\lf\{\sum_{k\in\nn}2^{-kp_-(\delta+\frac{n}{r}-\frac{n}{s})}
\lf\|\lf\{\sum_{j\in\nn}\lf[\cm\lf(\frac{|\lz_j|^s}{\|\chi_{B_j}\|^s_{\vp}}\chi_{B_j}\r)
\r]^{\frac{p_-}{s}}\r\}^{\frac{s}{p_-}}\r\|^{\frac{p_-}{s}}
_{L^{\frac{p(\cdot)}{s}}(\rn)}\r\}^{1/p_-}\\
&\ls\lf\{\sum_{k\in\nn}2^{-kp_-(\delta+\frac{n}{r}-\frac{n}{s})}
\lf\|\lf\{\sum_{j\in\nn}\lf[\frac{|\lz_j|^s}{\|\chi_{B_j}\|^s_{\vp}}\chi_{B_j}
\r]^{\frac{p_-}{s}}\r\}^{\frac{s}{p_-}}\r\|^{\frac{p_-}{s}}
_{L^{\frac{p(\cdot)}{s}}(\rn)}\r\}^{1/p_-}\\
&\sim\lf\{\sum_{k\in\nn}2^{-kp_-(\delta+\frac{n}{r}-\frac{n}{s})}
\lf\|\sum_{j\in\nn}\lf[\frac{|\lz_j|}{\|\chi_{B_j}\|_{\vp}}\chi_{B_j}
\r]^{p_-}\r\|_{L^{\frac{p(\cdot)}{p_-}}(\rn)}\r\}^{1/p_-}\\
&\ls\lf\{\sum_{k\in\nn}2^{-kp_-(\delta+\frac{n}{r}-\frac{n}{s})}
\ca(\{\lz_j\}_{j\in\nn},\{B_j\}_{j\in\nn})^{p_-}\r\}^{1/p_-}
\ls\ca(\{\lz_j\}_{j\in\nn},\{B_j\}_{j\in\nn}).
\end{align*}
By an argument similar to that used in the estimate of ${\rm K}$,
we also have ${\rm J}\ls\ca(\{\lz_j\}_{j\in\nn},\{B_j\}_{j\in\nn})$.
Hence,
$$\|f\|_{\ath}\ls\ca(\{\lz_j\}_{j\in\nn},\{B_j\}_{j\in\nn}).$$
This, combined with \eqref{eq 2.7x}, implies that
$\|f\|_{\ath}\ls\|f\|_{\vhp}$.
We complete the proof of Proposition \ref{pro-3}.
\end{proof}

We are now in a position to prove Theorem \ref{thm-2}.
\begin{proof}[Proof of Theorem \ref{thm-2}]
Let $\mu_0$ be as in \eqref{eq mu}, $p(\cdot)\in C^{\log}(\rn)$
with $\frac{n}{n+\mu_0}<p_-\le p_+\le1$ and $r\in(1,\,\fz)$ such that
$\mu_0+\frac{n}{r}>\frac{n}{p_-}$.
Then, by Proposition \ref{pro-2} and \ref{pro-3}, we know that,
$\ath$ and $\vhp$ coincide with equivalent quasi-norms. This finishes the proof
of Theorem \ref{thm-2}.
\end{proof}

\subsection{Proof of Corollary \ref{cor-1}}
In this subsection, we prove Corollary \ref{cor-1}.
We begin with introducing the atomic characterization of $\hp$.
\begin{definition}[\cite{ns12,cw14}]\label{def atom-1}
Let $p(\cdot)\in\cp(\rn)$, $r\in (p_+,\fz]\cap[1,\fz]$ and $s:=\lfloor\frac{n}{p_-}-n \rfloor$.
A measurable function $a$ on $\rn$ is called a $(p(\cdot),\,r,\,s)$-atom
associated with ball $B$ of $\rn$ if
\begin{enumerate}
\item[(i)] $\supp a\st B$;

\item[(ii)] $\|a\|_{L^r(\rn)}\le |B|^{\frac1r}\|\chi_B\|^{-1}_{\vp}$;

\item[(iii)] for any $\az\in\zz_+^n$ with $|\az|\le s$, $\int_\rn a(x)x^\az\,dx=0$.
\end{enumerate}
\end{definition}

The following lemma is just \cite[Theorem 1.1]{s13}, which
establishes the atomic characterization of $H^{p(\cdot)}(\rn)$ (see also \cite{ns12}).

\begin{lemma}[\cite{s13}]\label{lem 5.31}
Let $p(\cdot)\in C^{\log}(\rn)$ with $p_+\in(0,\,1]$.

{\rm (i)} Let $r\in(p_+,\,\fz]\cap[1,\,\fz]$ and $s:=\lfloor \frac{n}{p_-}-n\rfloor$.
Then there exists a positive constant $C$ such that,
for any $\{\lz_j\}_{j\in\nn}\st\cc$ and any family $\{a_j\}_{j\in\nn}$
of $(p(\cdot),\,r,\,s)$-atoms, associated with balls $\{B_j\}_{j\in\nn}$ of $\rn$, such that
$\ca(\{\lz_j\}_{j\in\nn},\,\{B_j\}_{j\in\nn})<\fz$, it holds true that
$f:=\sum_{j\in\nn}\lz_ja_j$ converges in $H^{p(\cdot)}(\rn)$ and
$\|f\|_{H^{p(\cdot)}(\rn)}\le C\ca(\{\lz_j\}_{j\in\nn},\,\{B_j\}_{j\in\nn}).$

{\rm (ii)} Let $s\in\zz_+$. For any $f\in H^{p(\cdot)}(\rn)$,
there exists a decomposition
$f=\sum_{j=1}^\fz\lz_ja_j$ in $\cs'(\rn),$
where $\{\lz_j\}_{j\in\nn}\st\cc$ and $\{a_j\}_{j\in\nn}$ is a family
of $(p(\cdot),\,\fz,\,s)$-atoms associated with balls $\{B_j\}_{j\in\nn}$ of $\rn$.
Moreover, there exists a positive constant $C$ such that, for any $f\in H^{p(\cdot)}(\rn)$,
\begin{align*}
\ca(\{\lz_j\}_{j\in\nn},\,\{B_j\}_{j\in\nn})\le C\|f\|_{H^{p(\cdot)}(\rn)}.
\end{align*}
\end{lemma}

We now prove Corollary \ref{cor-1}.
\begin{proof}[Proof of Corollary \ref{cor-1}]
By Lemma \ref{lem 5.31}(ii) and its proof,
we find that $H^{p(\cdot)}(\rn)\cap L^2(\rn)$ is dense in $H^{p(\cdot)}(\rn)$
and, for any $f\in [H^{p(\cdot)}(\rn)\cap L^2(\rn)]$, there exists a sequence
$\{\lz_j\}_{j\in\nn}\st\cc$ and a family $\{a_j\}_{j\in\nn}$
of $(p(\cdot),\,\fz,\,s)$-atoms,
associated with balls $\{B_j\}_{j\in\nn}$ of $\rn$,
such that $f=\sum_{j\in\nn}\lz_j a_j$ in $L^2(\rn)$
and $\ca(\{\lz_j\}_{j\in\nn},\,\{B_j\}_{j\in\nn})\ls\|f\|_{\hp}$.
Notice that each $(p(\cdot),\,\fz,\,0)$-atom
is also an $(p(\cdot),\,r)$-atom for any $r\in(1,\,\fz]$ (see Definition \ref{def atom}).
From this and the definition of $\ath$,
we further deduce that $f\in\ath$ and
$$\|f\|_{\ath}\le\ca(\{\lz_j\}_{j\in\nn},\,\{B_j\}_{j\in\nn})\ls\|f\|_{\hp}.$$
By this, Theorem \ref{thm-2} and an density argument,
we know that, for any $f\in\hp$,
$$\|f\|_{\vhp}\ls\|f\|_{\hp}.$$
This finishes the proof of Corollary \ref{cor-1}.
\end{proof}

\noindent
{\bf Acknowledgments.} Junqiang Zhang is very grateful to his advisor
Professor Dachun Yang for his guidance and encouragements.
Junqiang Zhang is
supported by the Fundamental Research Funds for the Central Universities
(Grant No. 2018QS01) and the National
Natural Science Foundation of China (Grant No. 11801555).
Zongguang Liu is supported by the National
Natural Science Foundation of China (Grant No. 11671397).

\bibliographystyle{amsplain}

\end{document}